\documentclass[11pt]{amsart}
\usepackage{amscd, amssymb}
\usepackage{graphics}

\theoremstyle{plain}

\numberwithin{equation}{section}

\newtheorem{proposition}[equation]{Proposition}
\newtheorem{theorem}{Theorem}
\Huge
\newtheorem{lemma}[equation]{Lemma}
\newtheorem{corollary}[equation]{Corollary}
\theoremstyle{definition}
\newtheorem{remark}[equation]{Remark}
\newtheorem{example}[equation]{Example}

\def    \R  {{\mathbb R}}
\def    \Z  {{\mathbb Z}}
\def    \Q  {{\mathbb Q}}

\def    \CP {{\mathbb {CP}}}

\def    \C  {{\mathbb C}}

\def    \index  {{\operatorname{index}}}
\def  \deg {{\operatorname{deg}}}

\def    \Tilde  {\widetilde}

\def    \ut     {\Tilde{u}}

\def    \Gt     {\Tilde{G}}

\begin{document}
\title[Ham circle actions with almost minimal isolated fixed points] {Hamiltonian circle actions with almost minimal isolated fixed points}

\author{Hui Li} 
\address{School of Mathematical Sciences\\
                Soochow University, Suzhou\\
                215006, China}
        \email{hui.li@suda.edu.cn}

\thanks{2010 classification.  53D05, 53D20, 55N25, 57R20, 32H02} 
\thanks{The author is supported by the NSFC grant K110712116.}
\begin{abstract}
Let the circle act in a Hamiltonian fashion on a connected compact symplectic manifold $(M, \omega)$ of dimension $2n$. Then the $S^1$-action has at least $n+1$ fixed points. In a previous paper, we study the case when the fixed point set consists of precisely $n+1$ isolated points.
In this paper, we study the case when the fixed point set  consists of exactly $n+2$ isolated points. We show that in this case $n$ must be even. We find equivalent
conditions on the first Chern class of $M$ and a particular weight of the $S^1$-action. 
We also show that the particular weight  can completely determine the integral cohomology ring and the total Chern class of $M$, and the sets of weights of the $S^1$-action at all the fixed points. We will see that all these data are isomorphic to those of known examples,  $\Gt_2(\R^{n+2})$ with $n\geq 2$ even, equipped with standard circle actions. 
\end{abstract}

 \maketitle

\section{Introduction}
Consider a circle action on a connected compact  $2n$-dimensional symplectic manifold $(M, \omega)$ with moment map $\phi$. By Morse theory, the fixed point set $M^{S^1}$ contains
at least $n+1$ points. In \cite{L2}, we study the case when $M^{S^1}$ consists of exactly $n+1$ isolated points.
In this paper, we consider the case when $M^{S^1}$ consists of exactly $n+2$ isolated points, in which case, we call the action has {\bf almost minimal}  isolated fixed points.
Such known examples  are $\Gt_2(\R^{n+2})$ with $n\geq 2$ even, see  Example~\ref{grasse}.

Let $(M, \omega)$ be a connected compact Hamiltonian $S^1$-manifold of dimension $2n$ with almost minimal isolated fixed points, and let $\phi$ be the moment map. By Lemma~\ref{order}, the fixed points, denoted 
$P_0$,  $\cdots$, $P_{\frac{n}{2}-1}$, $P_{\frac{n}{2}}$, $P_{(\frac{n}{2})'}$, $P_{\frac{n}{2}+1}$, $\cdots$, $P_n$, can be labelled to satisfy
$$\phi(P_0)  < \cdots < \phi(P_{\frac{n}{2}-1}) < \phi(P_{\frac{n}{2}})\leq\phi(P_{(\frac{n}{2})'})<\phi(P_{\frac{n}{2}+1}) < \cdots < \phi(P_n).$$
In the statements of our main theorems below, we are referring to this order
of moment map values of the fixed points.

For a symplectic $S^1$-manifold $M$ with isolated fixed points, in a neighborhood of each fixed point $P$, the $S^1$-action is equivalent to an $S^1$ linear action 
on $T_P M$, so there is a set of non-zero integers, called the {\it weights} of the $S^1$-action at the fixed point $P$. 

Our main results are as follows.
\begin{theorem}\label{equiv}
Let the circle act on a connected compact $2n$-dimensional symplectic manifold $(M, \omega)$ with moment map $\phi\colon M\to\R$. Assume $[\omega]$ is a
primitive integral class and the fixed point set $M^{S^1}$ consists of $n+2$ isolated points, denoted $M^{S^1}=\big\{P_0,  \cdots, P_{\frac{n}{2}-1}, P_{\frac{n}{2}}, P_{(\frac{n}{2})'}, P_{\frac{n}{2}+1}, \cdots, P_n\big\}$. Then for any $i, j\in \big\{0, \cdots, \frac{n}{2}, (\frac{n}{2})', \cdots, n\big\}$, $\phi(P_i)-\phi(P_j)\in \Z$, and $c_1(M)= n[\omega]$ if and only if $\phi(P_{n-1})-\phi(P_0)=\phi(P_n)-\phi(P_1)$ holds and this integer occurs as a weight of the $S^1$-action at some fixed point.  In dimension $4$, for the ``if" part to hold, the class $[\omega]$ needs to be chosen suitably.        
\end{theorem}

\begin{theorem}\label{lw}
Let the circle act on a connected compact $2n$-dimensional symplectic manifold $(M, \omega)$ with moment map $\phi\colon M\to\R$. Assume $M^{S^1}$ consists of $n+2$ isolated points, i.e.,
$M^{S^1}=\big\{P_0,  \cdots, P_{\frac{n}{2}-1}, P_{\frac{n}{2}}, P_{(\frac{n}{2})'}, P_{\frac{n}{2}+1}, \cdots, P_n\big\}$. Then $n\geq 2$ must be even. 
Assume $[\omega]$ is an integral class.
If $\phi(P_{n-1})-\phi(P_0)=\phi(P_n)-\phi(P_1)$ holds and this integer occurs as a weight of the $S^1$-action at some fixed point, then all the following are true.
\begin{enumerate}
\item  The integral cohomology ring of $M$ is isomorphic to that of $\Gt_2(\R^{n+2})$.
\item  The total Chern class of $M$ is isomorphic to that of $\Gt_2(\R^{n+2})$.
\item  The sets of weights of the $S^1$-action at all the fixed points on $M$ are isomorphic to those of a standard circle action on $\Gt_2(\R^{n+2})$ (as in Example~\ref{grasse}).
\end{enumerate}
\end{theorem}

Theorem~\ref{lw} follows from Propositions~\ref{01i2}, \ref{weights-ring} and \ref{weights-tc}. 

\medskip

By  \cite[Prop. 4.2 and Sec. 5]{L}, if the manifold $M$ in Theorem~\ref{lw} is K\"ahler, and the $S^1$-action is holomorphic, then $M$ is $S^1$-equivariantly biholomorphic and $S^1$-equivariantly symplectomorphic to $\Gt_2(\R^{n+2})$ with $n\geq 2$ even.

In \cite[Theorem 6.17]{H}, Hattori studies compact {\it almost complex} $S^1$-manifold $M$ of dimension $2n$  with $n+2$ isolated fixed points. 
Under the assumption $c_1(M)=nx$, where $x\in H^2(M; \Z)$ is a generator, 
and under additional technical conditions, Hattori shows that the sets of weights at all the fixed points are isomorphic to those of a standard circle action on $\Gt_2(\R^{n+2})$ with $n\geq 2$ even, and he obtains
that $x^n [M] = 2$.  In our current work, for a compact symplectic Hamiltonian $S^1$-manifold $(M, \omega)$, we prove the equivalence of the condition $c_1(M)=nx$ and the particular weight as in Theorem~\ref{equiv}, and using the particular weight as a starting point, we give methods to prove $(1)$, $(2)$ and $(3)$ in Theorem~\ref{lw}.

In \cite{L1}, the author studies compact Hamiltonian $S^1$-manifolds with fixed point set consisting of two connected components and almost minimal in a certain sense. Recent related works on compact Hamiltonian $S^1$-manifolds with fixed point set minimal in a certain sense include \cite{{T}, {M}, {LT}, {L}, {L2}, {LOS}, {GS}, {JT}}.

\

The paper is organized as follows. In Section~\ref{almostmin}, we give examples of Hamiltonian $S^1$-manifolds with almost minimal isolated fixed points,  prove and present some basic and key results.  In Section~\ref{c1larg}, we prove Theorem~\ref{equiv}. 
In Section~\ref{weight-all}, we determine
the sets of weights of the $S^1$-action at all the fixed points, proving $(3)$ of Theorem~\ref{lw}. In Section~\ref{weightringc},  we determine the integral cohomology ring and the total Chern class of the manifold, proving
$(1)$ and $(2)$ of Theorem~\ref{lw}.

\section{examples and basic results}\label{almostmin}

In this section, we give examples, prove and present key basic results.

\begin{example}\label{grasse}
Let $\Gt_2(\R^{n+2})$ be the  Grassmanian of oriented $2$-planes in $\R^{n+2}$, with $n\geq 2$ even. This $2n$-dimensional manifold naturally arises as a coadjoint orbit of $SO(n+2)$, hence it has a natural K\"ahler structure and a Hamiltonian $SO(n+2)$ action.

Consider the $S^1\subset SO(n+2)$ action on $\Gt_2(\R^{n+2})$ induced by the $S^1$ action on
$\R^{n+2}= \C^{\frac{n+2}{2}}$ given by
$$\lambda\cdot \big(z_0, z_1, \cdots, z_{\frac{n}{2}}\big) = \big(\lambda^{b_0}z_0, \lambda^{b_1}z_1, \cdots, \lambda^{b_{\frac{n}{2}}}z_{\frac{n}{2}}\big),$$
where the $b_i$'s, with $i = 0, 1, \cdots, \frac{n}{2}$, are mutually {\it distinct} integers.
This action has $n+2$ isolated fixed points, denoted 
$$P_0,\, P_1,\, \cdots,\, P_{\frac{n}{2}-1},\, P_{\frac{n}{2}},\,  P_{(\frac{n}{2})'},\,  P_{\frac{n}{2}+1},\, \cdots,\,  P_n,$$ 
where for each $0\leq i\leq \frac{n}{2}$,  $P_i$ and $P_{n-i}$ are given by the plane $(0, \cdots,  z_i,  \cdots, 0)$ respectively with two different orientations. Here, we use the convention $n-\frac{n}{2}=(\frac{n}{2})'$.
Let $\phi$ be the moment map of this $S^1$-action. Then the $\phi(P_i)$'s are respectively
$$-b_0,\, \cdots, \,-b_{\frac{n}{2}}, \, b_{\frac{n}{2}}, \, \cdots, \, b_0,$$ assuming in the order of nondecreasing. Note that if $b_{\frac{n}{2}}\neq 0$, then 
$\phi(P_{\frac{n}{2}}) < \phi(P_{(\frac{n}{2})'})$, otherwise, $\phi(P_{\frac{n}{2}}) = \phi(P_{(\frac{n}{2})'})$.
The set of weights of the $S^1$-action at any $P_i$, where $i\in\big\{0,  \cdots, \frac{n}{2}, (\frac{n}{2})', \cdots, n\big\}$, is
$$\big\{w_{ij} = \phi(P_j)-\phi(P_i)\big\}_{j\neq i, n-i}.$$

The ring $H^*\big(\Gt_2(\R^{n+2}); \Z\big)$ is generated by 
$x$, $y$ and  $z$, where $\deg(x) = 2$ and $\deg(y) = \deg(z) = n$. The relations are $x^{\frac{n}{2}}=y+z$, $xy=xz= \frac{1}{2}x^{\frac{n}{2}+1}$, and $y^2=z^2 = \frac{1+(-1)^{\frac{n}{2}}}{2}x^{\frac{n}{2}}y=\frac{1+(-1)^{\frac{n}{2}}}{2}x^{\frac{n}{2}}z$.
The total Chern class $c\big(\Gt_2(\R^{n+2})\big) =\frac{(1+x)^{n+2}}{1+2x}$, in particular, $c_1\big(\Gt_2(\R^{n+2})\big) = nx$.
\end{example}

Next, we prove a key basic lemma.

\begin{lemma}\label{order}
Let the circle act on a connected compact $2n$-dimensional symplectic manifold $(M, \omega)$ with moment map $\phi\colon M\to\R$. Assume $M^{S^1}$ consists of $n+2$ isolated points. Then $n$ must be even, and we can denote                  
$M^{S^1}=\{P_0,  \cdots, P_{\frac{n}{2}-1}, P_{\frac{n}{2}}, P_{(\frac{n}{2})'}, P_{\frac{n}{2}+1}, \cdots, P_n\},$ 
where the points can be labelled so that $P_i$ has Morse index $2i$ for all $i$, and  $P_{(\frac{n}{2})'}$ has Morse index $n$. Moreover, the  cohomology groups of $M$ are
  \[H^k(M; \Z)  = \left\{ \begin{array}{ll}
           \Z,   &  \mbox{if $k$  is even, $0\leq k\leq 2n$,  and  $k\neq n$},\\
           \Z\oplus \Z,  &    \mbox{if $k=n$},\\
           0,    & \mbox{if $k$  is odd},
           \end{array} \right. \]
and
\begin{equation}\label{eqorder}
 \phi(P_0)  < \cdots < \phi(P_{\frac{n}{2}})\leq\phi(P_{(\frac{n}{2})'})<\phi(P_{\frac{n}{2}+1}) < \cdots < \phi(P_n).
\end{equation}
\end{lemma}

\begin{proof}
Since $1, [\omega], \cdots, [\omega]^n$ are nonzero cohomology classes of $M$,
the even Betti numbers $b_{2i}(M)\geq 1$ for all $0\leq 2i\leq 2n$. Since $\phi$ is a perfect Morse function,
there is at least one fixed point, namely $P_i$, of index $2i$ for each $0\leq 2i\leq 2n$, which contributes to $b_{2i}$. By assumption, there is a remaining fixed point, let $2k$ be its Morse index. Since $b_{2i}(M)= b_{2n-2i}(M)$ for all $0\leq 2i\leq 2n$ by Poincar\'e duality,  we must have $2k=2n-2k$, which gives $2k=n$, i.e., $n$ is even. We denote the remaining fixed point by $P_{(\frac{n}{2})'}$, which has Morse index $n$. 
By Morse theory, $M$ has a natural CW-structure, with cells given by the negative disk bundles of the fixed points. By cellular cohomology theory, the groups $H^k(M; \Z)$ 
are as claimed. 

By  \cite[Lemma 3.1]{T} (cited as Lemma 2.3 in \cite{L2}), 
for each $P_i$, $\index(P_i)\leq 2l$, where $l$ is the number of fixed points
below $P_i$, we obtain
$\phi(P_0) < \phi(P_1) < \cdots < \phi(P_{\frac{n}{2}})\leq\phi(P_{(\frac{n}{2})'})$. Similarly, using $-\phi$, we obtain
$-\phi(P_n) < -\phi(P_{n-1}) < \cdots < -\phi(P_{(\frac{n}{2})'})\leq -\phi(P_{\frac{n}{2}})$.
The two inequalities give (\ref{eqorder}).
\end{proof}

\begin{remark}
From now on, when we use the notation 
$$M^{S^1}=\big\{P_0,  \cdots, P_{\frac{n}{2}-1}, P_{\frac{n}{2}}, P_{(\frac{n}{2})'}, P_{\frac{n}{2}+1}, \cdots, P_n\big\},$$ we mean that each point has the
Morse index as in Lemma~\ref{order}, and these points satisfy (\ref{eqorder}). We will use these facts tacitly. 
\end{remark}

Next, we get a basis of the equivariant and ordinary cohomology of the manifold.
When we speak about $S^1$-equivariant cohomology, we use $t\in H^2(\CP^{\infty}, \Z)$ to denote a generator.
For basic material in equivariant cohomology, we refer to \cite{L2}.

\begin{proposition}\label{basis}
Let $(M, \omega)$ be a connected compact $2n$-dimensional Hamiltonian $S^1$-manifold with moment map 
$\phi\colon M\to\R$. Assume 
$$M^{S^1}=\big\{P_0, \cdots, P_{\frac{n}{2}-1}, P_{\frac{n}{2}}, P_{(\frac{n}{2})'}, P_{\frac{n}{2}+1}, \cdots, P_n\big\}.$$ 
Let  $I=\big\{0, \cdots, \frac{n}{2}, (\frac{n}{2})', \cdots, n\big\}$. 
Then as an
$H^*(\CP^{\infty}; \Z)$-module, $H^*_{S^1}(M; \Z)$ has a basis 
$\big\{\Tilde\alpha_i \,|\, i\in I\big\}$ such that for any $i\in I$,
$$\Tilde\alpha_i|_{P_i} = \Lambda^-_i t^i, \,\,\, \Tilde\alpha_i|_{P_j} =0, \,\,\,\forall \,\, P_j \,\,\,\mbox{with}\,\,\, \index (P_j)\leq\index (P_i),$$
where $\Lambda^-_i$ is the product of the negative weights at $P_i$.
Moreover,  $\big\{\alpha_i = r(\Tilde\alpha_i)\,|\, i\in I\big\}$ is a basis for $H^*(M; \Z)$, where $r\colon H^*_{S^1}(M; \Z)\to H^*(M; \Z)$ is the natural restriction.
\end{proposition}

\begin{proof}
 Using standard argument in equivariant cohomology,
we get the basis of $H^*_{S^1}(M; \Z)$ satisfying the other conditions except that $\Tilde\alpha_{\frac{n}{2}}|_{P_{(\frac{n}{2})'}}=0$.  If $\Tilde\alpha_{\frac{n}{2}}|_{P_{(\frac{n}{2})'}} = a\Tilde\alpha_{(\frac{n}{2})'}|_{P_{(\frac{n}{2})'}}\neq 0$, where $a\in\Z$, then replace $\Tilde\alpha_{\frac{n}{2}}$ by  $\Tilde\alpha_{\frac{n}{2}}-a\Tilde\alpha_{(\frac{n}{2})'}$, still denoting it $\Tilde\alpha_{\frac{n}{2}}$.

    Since the fixed points are isolated, hence has no torsion cohomology, 
the restriction map $r$ is onto (\cite[Sec. 2]{LT}), so $\big\{\alpha_i = r(\Tilde\alpha_i)\,|\, i\in I\big\}$ is a basis for $H^*(M; \Z)$.
\end{proof}

\begin{corollary}\label{cor2}
Let $(M, \omega)$ be a connected compact $2n$-dimensional Hamiltonian $S^1$-manifold with moment map 
$\phi\colon M\to\R$. Assume 
$$M^{S^1}=\big\{P_0, \cdots, P_{\frac{n}{2}-1}, P_{\frac{n}{2}}, P_{(\frac{n}{2})'}, P_{\frac{n}{2}+1}, \cdots, P_n\big\}.$$ 
If $\Tilde\alpha\in H^{2k}_{S^1}(M; \Z)$ is a class such that $\Tilde\alpha|_{P_i}=0$
for all $P_i$'s with $\phi(P_i) < a$,  then 
$$\Tilde\alpha = \sum_i a_i \Tilde\alpha_i,$$
where the sum is over the indices $i$'s such that $\phi(P_i)\geq a$ and $\deg(\Tilde\alpha_i)\leq 2k$, and $a_i\in H^*(\CP^{\infty}; \Z)$.  
Here, $i\in I = \big\{0, \cdots, \frac{n}{2}, (\frac{n}{2})', \cdots, n\big\}$.
\end{corollary}

Next, we state some known results which we will use in the next sections.

\begin{lemma}\cite[Lemmas 2.1 and 2.2]{L2}\label{ut}
Let the circle act  on a connected compact symplectic manifold $(M,\omega)$ with moment map $\phi \colon M \to \R$.  Assume the fixed points are isolated. Then there exists $\ut =[\omega -\phi t]\in H_{S^1}^2(M;\R)$ such that for any fixed point $P$,
$$\ut|_P = - \phi(P)t.$$

If $[\omega]$ is an integral class, then $\ut$ is an integral class.  In this case, for any two fixed points $P$ and $Q$, $\phi(P) - \phi(Q) \in \Z$.  
If $\Z_k$ is the stabilizer group of some point on $M$,  then for any two fixed points $P$ and
$Q$ on the same connected component of $M^{\Z_k}$, we have $k\,|\left(\phi(P) - \phi(Q)\right)$.
\end{lemma}

In Lemma~\ref{ut}, $M^{\Z_k}$ is the {\bf $\Z_k$-isotropy submanifold}, 
i.e., the set of points pointwise fixed by $\Z_k$ but not pointwise fixed by $S^1$.

By Lemma~\ref{ut}, if $[\omega]$ is an integral class, then any weight
of the $S^1$-action is no larger than the length of the interval $\phi(M)$.

\begin{theorem}\cite{AB, BV}\label{AB.BV}
Let the circle act on a compact oriented manifold $M$. Assume the fixed points are isolated.
Fix a class $\alpha\in H^*_{S^1}(M; \Q)$. Then as elements of $\Q(t)$,
$$\int_M \alpha = \sum_{P\subset M^{S^1}} \frac{\alpha|_P}{e^{S^1}(N_P)},$$
where the sum is over all the fixed points, and $e^{S^1}(N_P)$ is the equivariant Euler
class of the normal bundle to $P$.
\end{theorem}

\begin{lemma}\cite[Lemma 2.6]{T}\label{equalmod}
Let the circle act on a compact symplectic manifold $(M, \omega)$. Let $p$ and $q$ be fixed points which lie on the same connected component of $M^{\Z_k}$ for some $k > 1$. Then the $S^1$ weights at $p$ and $q$ are equal modulo $k$.
\end{lemma}

\section{on the first Chern class of the manifold  ---  the proof of Theorem \ref{equiv}}\label{c1larg}

In this section, we prove Theorem  \ref{equiv}.

In a symplectic $S^1$-manifold $(M, \omega)$ with isolated fixed points,
if $w>0$ is a weight of the $S^1$-action at a fixed point $P$, $-w$ is a weight of the $S^1$-action at a fixed point $Q$, and $P$ and $Q$ are on the same connected component of
$M^{\Z_w}$, then we say that {\bf there is a weight $w$ from $P$ to $Q$}, or $w$ is a weight from $P$ to $Q$. 
When the signs of $w$ at $P$ and at $Q$ are clear, we also say that {\bf there is a weight $\pm w$ between $P$ and $Q$}. 
It is known (see \cite{H} for example) that the set $W^+$ of all the positive weights and the set $W^-$ of all the negative weights at all the fixed points  satisfy $W^- = - W^+$. We will use this fact tacitly.

\smallskip

The following Lemma~\ref{mlarge} is a consequence of Lemmas 4.2 and 4.4 in \cite{L2}.
\begin{lemma}\label{mlarge}
Let the circle act on a connected compact symplectic manifold $(M, \omega)$ with moment map $\phi\colon M\to\R$. Assume 
$$M^{S^1}=\big\{P_0, \cdots, P_{\frac{n}{2}-1}, P_{\frac{n}{2}}, P_{(\frac{n}{2})'}, P_{\frac{n}{2}+1}, \cdots, P_n\big\}.$$ 
 Let $w>0$ be the largest in the set of all the weights of the $S^1$-action at all the fixed points. Then there are $P_i$ and $P_j$ with $\index(P_i)\leq\index(P_j)$, such that $w$ is a weight from $P_i$ to $P_j$. Here,
$i, j\in I = \big\{0, \cdots, \frac{n}{2}, (\frac{n}{2})', \cdots, n\big\}$. 
If $c_1(M) = k [\omega]$,  then $\frac{1}{2}\index(P_j) -\frac{1}{2}\index(P_i) +1 = k\frac{\phi(P_j)-\phi(P_i)}{w}$. 
\end{lemma}

We will use Lemma~\ref{lh} in the proof of Lemma~\ref{4dim}.
\begin{lemma}\label{lh}
Let the circle act on a connected compact symplectic manifold $(M, \omega)$ with moment map $\phi\colon M\to\R$. Assume 
$M^{S^1}$ consists of isolated points. Then for any index $2$ fixed point, there is a weight between it and a fixed point below it (relative to the value of $\phi$). Similarly, for any coindex $2$ fixed point, there is a weight between it and a fixed point above it.
\end{lemma}

\begin{proof}
Let $P\in M^{S^1}$ be of index $2$ and $-w$ be the negative weight at $P$. 
Since $P$ has only one negative weight, the connected component $C$ of $M^{\Z_w}$ containing $P$ has $P$ as an index $2$ fixed point.
Since $C$ is compact and symplectic, it contains a unique fixed point $Q$ as the minimum of $\phi|_C$, so $\phi(Q) < \phi(P)$. Note that $w$ is the smallest positive weight on $C$ (other
weights on $C$ must be multiples of $w$), 
using \cite[Proposition 3]{JT} (cited as Lemma 5.6 in \cite{L2}) on $C$, we get that $w$ is a positive weight at $Q$.  If $w=1$, we regard $C=M$. The other claim follows similarly by using $-\phi$.
\end{proof}

To prove Theorem~\ref{equiv}, note that when $\dim(M)=4$, $\dim H^2(M)$ is $2$. In this case, we need to choose the class
$[\omega]$ suitably so that $c_1(M)$ is a multiple of $[\omega]$. We consider this case separately, at the same time, we obtain the sets of weights at all the fixed points.

\begin{lemma}\label{4dim}
Let the circle act on a connected compact $4$-dimensional symplectic manifold
$(M, \omega)$ with moment map $\phi\colon M\to\R$. Assume $[\omega]$ is an integral class and $M^{S^1}=\big\{P_0, P_1, P_{1'}, P_2\big\}$. 
Assume $\phi(P_{1'})-\phi(P_0)=\phi(P_2)-\phi(P_1)$ is a weight of the $S^1$-action at some fixed point. Then 
\begin{equation}\label{01=12}
\phi(P_1)-\phi(P_0)=\phi(P_2)-\phi(P_{1'}),
\end{equation}
and $[\omega]$ is primitive integral. Assume furthermore that $[\omega|_{S_1}]$ is
primitive integral, where $S_1$ is an invariant gradient sphere for any invariant metric from $P_1$ to $P_0$. Then  
the set of weights at any $P_i$  is
\begin{equation}\label{0112} 
\big\{w_{ij}=\phi(P_j)-\phi(P_i)\big\}_{j\neq i,\, 2-i},
\end{equation}
where $i, j \in \{0, 1, 1',  2\}$, and we use the convention $1+1' =2$; moreover, $c_1(M)=2[\omega]$.
\end{lemma}

\begin{proof}
The equality $\phi(P_{1'})-\phi(P_0)=\phi(P_2)-\phi(P_1)$ implies (\ref{01=12}).

Denote $w_{01'}=\phi(P_{1'})-\phi(P_0)$. If $\phi(P_1) < \phi(P_{1'})$, 
then $w_{01'}$ can only divide 
$\phi(P_{1'})-\phi(P_0)$, hence it can only be a weight between $P_0$ and $P_{1'}$.
If $\phi(P_1) = \phi(P_{1'})$, then by Lemma~\ref{lh}, each of $P_1$ and $P_{1'}$ has a weight with $P_0$, hence in this case $w_{01'}$ is also a weight between $P_0$ and $P_{1'}$ (not between $P_0$ and $P_2$, since there can only be two weights at $P_0$).  For any invariant metric, let $S_{1'}$ be
the gradient sphere from $P_{1'}$ to $P_0$ (if $\gamma$ is the gradient line from $P_{1'}$ to $P_0$, then $S^1\cdot \gamma$ is the gradient sphere), then  
$$\int_{S_{1'}} [\omega] =\frac{\phi(P_{1'})-\phi(P_0)}{w_{01'}}=1.$$ This means that 
$[\omega|_{S_{1'}}]$ is primitive integral, hence $[\omega]$ is primitive integral.
Let the weight between $P_1$ and $P_0$ be $w$. By assumption,  
$[\omega|_{S_1}]$ is primitive integral, similar to the above, we obtain that $|w| = \phi(P_1)-\phi(P_0)$. Denote $w_{01}= \phi(P_1)-\phi(P_0)$. By symmetry, or by using $-\phi$, similarly, we know that $w_{12}=\phi(P_2)-\phi(P_1)$ is a weight between $P_1$ and $P_2$, and $w_{1'2} = \phi(P_2)-\phi(P_{1'})$ is a weight between $P_{1'}$ and $P_2$.
The claim (\ref{0112}) follows.

Now consider the basis element $\Tilde\alpha_1$ and $\Tilde\alpha_{1'}$ of  $H^2_{S^1}(M; \Z)$
in Proposition~\ref{basis}. It is easy to check, by using (\ref{0112}) for $P_1$ and $P_{1'}$, that the restrictions of the degree $2$ classes $\ut +\phi(P_0)t$ and $\Tilde\alpha_1 + \Tilde\alpha_{1'}$ to the index $0$ and index $2$ fixed points $P_0$, $P_1$ and $P_{1'}$ are equal. Hence
$$\ut +\phi(P_0)t = \Tilde\alpha_1 + \Tilde\alpha_{1'}.$$
Restricting it to ordinary cohomology, we obtain $[\omega] = \alpha_1 + \alpha_{1'}$.
By the same argument, we obtain
$$c^{S^1}_1(M) = 2\Tilde\alpha_1 + 2 \Tilde\alpha_{1'} + \Gamma_0 t.$$
Restricting this to ordinary cohomology, we obtain $c_1(M)=2(\alpha_1 + \alpha_{1'})$.
So the claim $c_1(M) = 2 [\omega]$ follows.
\end{proof}

To prove Theorem~\ref{equiv}, let us first show that there is no weight
larger than the integer $\phi(P_{n-1})-\phi(P_0)=\phi(P_n)-\phi(P_1)$.

\begin{proposition}\label{0nnot}
Let the circle act on a connected compact symplectic manifold $(M, \omega)$ with moment map $\phi\colon M\to\R$. Assume $[\omega]$ is an integral class and
$M^{S^1}=\big\{P_0, \cdots, P_{\frac{n}{2}-1}, P_{\frac{n}{2}}, P_{(\frac{n}{2})'}, P_{\frac{n}{2}+1}, \cdots, P_n\big\}$.
Then the integer $\phi(P_n)-\phi(P_0)$ cannot be a weight at any fixed point.
\end{proposition}

\begin{proof}
Let $I=\big\{0, \cdots, \frac{n}{2}, (\frac{n}{2})', \cdots, n\big\}$. By Lemma~\ref{ut}, for any $i, j\in I$, $\phi(P_i)-\phi(P_j)\in\Z$. Suppose $w_{0n}=\phi(P_n)-\phi(P_0)$
is a weight at some fixed point. Since it only divides $\phi(P_n)-\phi(P_0)$,
it can only be a weight between $P_0$ and $P_n$, and it can only have
multiplicity $1$. Similar to the proof of \cite[Lemma 5.7]{L2}, using
$$\big\{\mbox{weights at $P_0$}\big\} = \big\{\mbox{weights at $P_n$}\big\} \mod w_{0n},$$
we obtain that the possible weights at $P_0$ and at $P_n$ are respectively in the sets 
\begin{equation}\label{set}
\big\{\phi(P_j)-\phi(P_0)\big\}_{j\in I, \,j\neq 0},\,\,\,\mbox{and}\,\,\, 
\big\{\phi(P_j)-\phi(P_n)\big\}_{j\in I, \,j\neq n}.
\end{equation}
Moreover, in the following two cases, $\phi(P_j)-\phi(P_0)$ is a weight at $P_0$ if and only if $\phi(P_j)-\phi(P_n)$ is a weight at $P_n$. 
\begin{enumerate}
\item  $\phi\big(P_{\frac{n}{2}}\big) < \phi\big(P_{(\frac{n}{2})'}\big)$, 
and $j\neq 0, n$. 
\item $\phi\big(P_{\frac{n}{2}}\big) = \phi\big(P_{(\frac{n}{2})'}\big)$, and
$j\neq 0, n, \frac{n}{2}, (\frac{n}{2})'$.
\end{enumerate}
Combining with a slight
modification of \cite[Lemma 5.1]{L2}, we can see that any weight occurs with multiplicity $1$. Since there are $n$ number of weights at $P_0$ and at $P_n$, and there are $n+1$ number of elements in each set in (\ref{set}), there exists $k\neq 0, n$ (also $k\neq 1, n-1$ if we use Lemma~\ref{lh})
such that $\phi(P_k)-\phi(P_0)$ does not occur as a weight
at $P_0$. By symmetry, $\phi(P_{n-k})-\phi(P_n)$ does not occur as a weight at $P_n$.
If $\phi\big(P_{\frac{n}{2}}\big) < \phi\big(P_{(\frac{n}{2})'}\big)$, or
if $\phi\big(P_{\frac{n}{2}}\big) = \phi\big(P_{(\frac{n}{2})'}\big)$ and
$k\neq \frac{n}{2}, (\frac{n}{2})'$, 
then by the fact above, $\phi(P_{n-k})-\phi(P_0)$ is also not a weight at $P_0$ (similarly $\phi(P_k)-\phi(P_n)$ is not a weight at $P_n$), so $P_0$ (and $P_n$) has less than $n$ number of weights, a contradiction. 
Now assume $\phi\big(P_{\frac{n}{2}}\big) = \phi\big(P_{(\frac{n}{2})'}\big)$,
and $k = (\frac{n}{2})'$ (similarly for $k=\frac{n}{2}$). It is conceivable that 
$\phi(P_{\frac{n}{2}})-\phi(P_0)$ occurs at $P_0$ and 
$\phi(P_{(\frac{n}{2})'})-\phi(P_0)$ does not occur at $P_0$, 
and $\phi(P_{(\frac{n}{2})'})-\phi(P_n)$ occurs at $P_n$ and 
$\phi(P_{\frac{n}{2}})-\phi(P_n)$ does not occur at $P_n$. We claim that this cannot happen. We look at it as follows. By symmetry, we may normalize $[\omega]$ and $\phi$ such that
$\phi(P_{\frac{n}{2}})-\phi(P_0) = \phi(P_n)-\phi(P_{\frac{n}{2}})$. Denote
$w = \phi(P_{\frac{n}{2}})-\phi(P_0)= \phi(P_n)-\phi(P_{\frac{n}{2}})$. Then $w_{0n} = 2 w$. Note that $w$, as a weight at 
$P_0$,  can only divide $\phi(P_{\frac{n}{2}})-\phi(P_0)$ and $\phi(P_n)-\phi(P_0)$. By \cite[Lemma 5.5]{L2}, $w$ is a weight between $P_0$ and 
$P_{\frac{n}{2}}$, and this number is a weight between $P_{(\frac{n}{2})'}$ and $P_n$. The isotropy submanifold $M^{\Z_w}$ containing $P_0$ contains also $P_{\frac{n}{2}}$, $P_n$ and $P_{(\frac{n}{2})'}$.  It is $4$-dimensional with exactly these $4$ fixed points, there is a weight
$w_{0n}$ between $P_0$ and $P_n$, and there is no weight between 
$P_{(\frac{n}{2})'}$ and $P_0$. By Lemma~\ref{lh} (used on $M^{\Z_w}$), the 
weights structure on this submanifold is not possible. 
\end{proof}

\begin{proof}[Proof of Theorem~\ref{equiv}]
First, by Lemma~\ref{ut}, we have $\phi(P_i)-\phi(P_j)\in\Z$  for any $i, j\in\{0, \cdots, \frac{n}{2}, (\frac{n}{2})', \cdots, n\}$.

If $\dim(M)=2n > 4$, by Lemma~\ref{order}, $H^2(M; \Z)=\Z$. So $c_1(M)=k[\omega]$ for some $k\in\Z$. Let $w>0$ be the largest weight of
the action at all the fixed points. By Lemma~\ref{mlarge},  $w$ is a weight from $P_i$ to $P_j$ with $\index(P_i)\leq\index(P_j)$. Then
$w|\big(\phi(P_j)-\phi(P_i)\big)$ by Lemma~\ref{ut}. By Lemma~\ref{mlarge}, $k=n$ implies that $w=\phi(P_{n-1})-\phi(P_0)=\phi(P_n)-\phi(P_1)$. Conversely, assume $\phi(P_{n-1})-\phi(P_0)=\phi(P_n)-\phi(P_1)$ holds and this integer occurs as a weight, by Proposition~\ref{0nnot}, this integer is the largest weight; since it only divides
$\phi(P_{n-1})-\phi(P_0)$ and $\phi(P_n)-\phi(P_1)$, it is a weight
between $P_0$ and $P_{n-1}$ (and is also a weight between $P_1$ and
$P_n$). By Lemma~\ref{mlarge}, $k=n$.

Now consider $\dim(M)=4$. If $c_1(M)=2[\omega]$, then the same argument as above yields that $\phi(P_{1'})-\phi(P_0)=\phi(P_2)-\phi(P_1)$ is the largest weight. For the converse, we do not know that $c_1(M)$ is a multiple of $[\omega]$;  the converse is by Lemma~\ref{4dim}. 
\end{proof}

\section{from the named weight to all the weights}\label{weight-all}

In this section, using the given weight, we find the sets of weights at all the fixed points, proving $(3)$ of Theorem~\ref{lw}.

For convenience, we state the following lemma.
\begin{lemma}\cite[Lemma 4.1]{L2}\label{sub}
Let the circle act on a connected compact symplectic manifold $(M, \omega)$ with moment map $\phi\colon M\to\R$. If $c_1(M)=k[\omega]$, then for any two fixed point set components
$F$ and $F'$,  we have $\Gamma_F - \Gamma_{F'} = k\big(\phi(F')-\phi(F) \big)$, where $\Gamma_F$ and $\Gamma_{F'}$ are respectively the sums of the weights at $F$ and $F'$.
\end{lemma}

We first find the sets of weights at $P_0$, $P_1$, $P_{n-1}$ and $P_n$.
\begin{lemma}\label{01n}
Let the circle act on a connected compact $2n$-dimensional symplectic manifold
$(M, \omega)$ with moment map $\phi\colon M\to\R$. Assume $[\omega]$ is an integral class and
$M^{S^1}=\big\{P_0, \cdots, P_{\frac{n}{2}-1}, P_{\frac{n}{2}}, P_{(\frac{n}{2})'}, P_{\frac{n}{2}+1}\cdots, P_n\big\}$. 
Assume $\phi(P_{n-1})-\phi(P_0)=\phi(P_n)-\phi(P_1)$ occurs as a weight of the $S^1$-action at some fixed point. Then 
\begin{equation}\label{01=nn-1}
\phi(P_1)-\phi(P_0) = \phi(P_n)-\phi(P_{n-1}),
\end{equation}
$[\omega]$ is primitive integral, and for $i=0$, $1$, $n-1$, $n$,
the set of weights at $P_i$ is
$$\big\{w_{ij}=\phi(P_j)-\phi(P_i)\big\}_{j\neq i, n-i}.$$
\end{lemma}

\begin{proof}
Lemma~\ref{4dim} gives the claims for $n=2$. So we assume now $n > 2$.
The fact $\phi(P_{n-1})-\phi(P_0)=\phi(P_n)-\phi(P_1)$ implies (\ref{01=nn-1}). 
Let $w_{0,n-1}=\phi(P_{n-1})-\phi(P_0)$. Since  $w_{0,n-1}$ only divides
$\phi(P_{n-1})-\phi(P_0)$, $w_{0,n-1}$ can only be a weight between $P_0$ and $P_{n-1}$ (and between $P_n$ and $P_1$). As a weight between any
pair of points, it can only have multiplicity $1$.

For the following steps, we proceed similarly as in the proof of Lemma 5.16
in \cite{L2}.  First we observe that the integration of $[\omega]$ on the isotropy sphere $M^{\Z_{w_{0, n-1}}}$ is $1$, so we
get that $[\omega]$ is primitive integral. Since $n > 2$, there is a unique index $2$ fixed point $P_1$ and a unique coindex $2$ fixed point $P_{n-1}$; the proof of \cite[Lemma 5.15]{L2} still works, and we get that $w_{01}=\phi(P_1)-\phi(P_0)$
is a weight between $P_0$ and $P_1$, and $w_{n-1, n}=\phi(P_n)-\phi(P_{n-1})$ is a weight
between $P_{n-1}$ and $P_n$. By Proposition~\ref{0nnot}, 
$\phi(P_n)-\phi(P_0)$ is not a weight, so $w_{0, n-1}$ is the largest weight.
As in the proof of \cite[Lemma 5.16]{L2}, we use 
$$\{\mbox{weights at $P_0$} \}=\{\mbox{weights at $P_{n-1}$}\} \mod w_{0, n-1}$$
to get the following {\it possibilities} about the weights at $P_0$ and $P_{n-1}$:
\begin{enumerate}
\item For $j\in\big\{1, \cdots, \frac{n}{2}, (\frac{n}{2})', \cdots, n-2\big\}$, there is a weight $w_{0j}=\phi(P_j)-\phi(P_0)$ between
$P_0$ and some $P_m$ with $m\in\big\{2, \cdots, \frac{n}{2}, (\frac{n}{2})', \cdots, n-2\big\}$, or between $P_0$ and $P_n$, and correspondingly, there is a weight $w_{n-1, j} =\phi(P_j)-\phi(P_{n-1})$ between $P_{n-1}$ and some $P_l$ with $l\in\big\{1, \cdots, \frac{n}{2}, (\frac{n}{2})', \cdots, n-2\big\}$. (The $w_{01}$ here is different from the one above between $P_0$ and $P_1$, or one may view it as it appeared with multiplicities.)
\item $\bar w_{0n} = \frac{1}{2}\big(\phi(P_n)-\phi(P_0)\big)$ is a weight
 between $P_0$ and $P_n$ with multiplicity $1$, and correspondingly, 
$\bar w_{n-1, 1}=\frac{1}{2}\big(\phi(P_1)-\phi(P_{n-1})\big)$ is a weight between $P_{n-1}$ and $P_1$ with multiplicity $1$.
\end{enumerate}
We have $3$ claims about these possibilities:

 Claim 1: For $j\in\big\{1,  \cdots, \frac{n}{2}, (\frac{n}{2})', \cdots, n-2\big\}$, if $w_{n-1, j}=\phi(P_j)-\phi(P_{n-1})$ appears as a weight at $P_{n-1}$, then it has multiplicity $1$. 

 Claim 2: $w_{n-1, 1}$ and  $\bar w_{n-1, 1}$ cannot appear at the same  time.

 Claim 3: Neither $w_{n-1, 1}$ nor $\bar w_{n-1, 1}$ can appear.

The proofs of Claims 1 and 2 are the same as in the proof of \cite[Lemma 5.16]{L2}  (Lemma 5.1 in \cite{L2} might be used with very minor modification when the two index $n$ fixed points have equal moment map values). Now we prove Claim 3. First assume $w_{n-1, 1}$ appears.  Then by Claims 1 and 2, $w_{n-1, 1}$ has multiplicity $1$ and $\bar w_{n-1, 1}$ does not appear, hence $(2)$ does not occur. Since there are $n-1$ number of negative weights at $P_{n-1}$, and there are $n$ number of fixed points below $P_{n-1}$, by $(1)$ and Claim 1, there exists one $k\in\big\{2, \cdots, \frac{n}{2}, (\frac{n}{2})', \cdots, n-2\big\}$ such that $w_{n-1, k}=\phi(P_k)-\phi(P_{n-1})$ is not a weight at $P_{n-1}$. By symmetry, or similarly by using $-\phi$, we get that $w_{1, n-k} = \phi(P_{n-k})-\phi(P_1)$ is not a (positive) weight at $P_1$. Then we can write down the sets of weights at $P_{n-1}$ and $P_1$:
$$\big\{w_{n-1, j}=\phi(P_j)-\phi(P_{n-1})\big\}_{j\neq n-1,\, k},$$
$$\big\{w_{1j}=\phi(P_j)-\phi(P_1)\big\}_{j\neq 1,\, n-k},\,\,\,\mbox{where $ w_{1, n-1} = - w_{n-1,1}$},$$
for some  $k\in\big\{2, \cdots, \frac{n}{2}, (\frac{n}{2})', \cdots, n-2\big\}$.
(Here $n-\frac{n}{2}=(\frac{n}{2})'$.)
We compute
\begin{equation}\label{n=} 
\Gamma_1 - \Gamma_{n-1}=\sum w_{1j} - \sum w_{n-1, j} \neq n\big(\phi(P_{n-1})-\phi(P_1)\big),
\end{equation} 
which contradicts to Lemma~\ref{sub} and Theorem~\ref{equiv}.  Hence $w_{n-1, 1}$ does not appear. (We could also use $P_0$ and $P_{n-1}$, or
$P_0$ and $P_n$ to do the argument above.)
Next, assume $\bar w_{n-1, 1}$ appears, i.e., $(2)$ occurs. Then $w_{n-1, 1}$ does not appear (the corresponding $w_{01}$ in $(1)$ does not appear). Similar to the above, let us write down the sets of weights at $P_0$ and $P_n$:
$$\big\{w_{0j}=\phi(P_j)-\phi(P_0)\big\}_{j\neq 0,\, n,\, k} \cup\, \bar w_{0n},$$
$$\big\{w_{nj}=\phi(P_j)-\phi(P_n)\big\}_{j\neq n,\, 0,\, n-k} \cup \bar w_{n0},\,\,\,\mbox{where $\bar w_{n0} = - \bar w_{0n}$},$$
for some  $k\in\big\{2, \cdots, \frac{n}{2}, (\frac{n}{2})', \cdots, n-2\big\}$. 
If $\phi(P_{\frac{n}{2}}) < \phi(P_{(\frac{n}{2})'})$, or if $\phi(P_{\frac{n}{2}}) = \phi(P_{(\frac{n}{2})'})$ and $k\neq \frac{n}{2}, (\frac{n}{2})'$, the same argument as above gives a contradiction. Now assume 
$\phi(P_{\frac{n}{2}}) = \phi(P_{(\frac{n}{2})'})$ and $k = (\frac{n}{2})'$
(or $k = \frac{n}{2}$). We argue similarly as in the last part of the proof
of Proposition~\ref{0nnot} to get a contradiction by looking at the weights structure on a $4$-dimensional isotropy submanifold containing $P_0$,
$P_{\frac{n}{2}}$, $P_{(\frac{n}{2})'}$ and $P_n$.
By all the facts above, the claims on the weights follow.
\end{proof}

Next, we find the sets of weights at all the fixed points.
\begin{proposition}\label{01i2}
Let the circle act on a connected compact $2n$-dimensional symplectic manifold
$(M, \omega)$ with moment map $\phi\colon M\to\R$. Assume $[\omega]$ is an integral class and
$M^{S^1}=\big\{P_0,  \cdots, P_{\frac{n}{2}-1}, P_{\frac{n}{2}}, P_{(\frac{n}{2})'}, P_{\frac{n}{2}+1}\cdots, P_n\big\}$.
Assume the integer $\phi(P_{n-1})-\phi(P_0)=\phi(P_n)-\phi(P_1)$ occurs as a weight of the $S^1$-action at some fixed point.  
Then the set of weights at any $P_i$ is
\begin{equation}\label{wi2}
\big\{w_{ij}=\phi(P_j)-\phi(P_i)\big\}_{j\neq i, n-i}.
\end{equation}
Moreover,
\begin{equation}\label{kn=2}
\phi(P_k)-\phi(P_0) = \phi(P_n)-\phi(P_{n-k}), \,\,\,\mbox{for any $k$ with $1\leq k\leq\frac{n}{2}$}.
\end{equation}
Hence the sets of weights at the fixed points are isomorphic to those of the standard circle action on $\Gt_2(\R^{n+2})$, with $n\geq 2$ even, as in Example~\ref{grasse}.
\end{proposition}

\begin{proof}
Lemma~\ref{4dim} gives the results for $n=2$. For $n>2$, 
Lemma~\ref{01n} gives (\ref{wi2}) for $P_0$, $P_{n-1}$, $P_n$ and $P_1$, and
(\ref{kn=2}) for $k=1$. By Lemma~\ref{ktoi} below, for any $i=0, 1, n-1, n$, $w_{ij}$
is a weight between $P_i$ and $P_j$ for each $j\neq i, n-i$.

The rest of the proof is similar to that of Proposition 6.10 in \cite{L2},
keeping in mind that $\phi(P_n)-\phi(P_0)$ is not a weight, by Proposition~\ref{0nnot}. We use induction. Fix $i$ with $2\leq i\leq \frac{n}{2}$. Assume that for each $k$ with $0\leq k < i$, (\ref{wi2}) holds for $P_k$ and $P_{n-k}$,  $w_{kj}$ is a weight between $P_k$ and $P_j$ for each $j\neq k, n-k$, and $w_{n-k, j}$ is a weight between $P_{n-k}$ and $P_j$ for each $j\neq n-k, k$, moreover, (\ref{kn=2}) holds for all $1\leq k \leq i-1$.
We need to prove all these hold if we replace $k$ by $i$. 
Proceed similarly as in the proof of Proposition 6.10 in \cite{L2}, and use Lemma~\ref{ktoi} below.
\end{proof}

Lemma~\ref{ktoi} can be proved similarly as Lemmas 6.1 and 6.9  
in \cite{L2}.
\begin{lemma}\label{ktoi}
 Let the circle act on a connected compact $2n$-dimensional symplectic manifold
$(M, \omega)$ with moment map $\phi\colon M\to\R$, where $n > 2$ is even.  
Assume $M^{S^1}=\big\{P_0,  \cdots, P_{\frac{n}{2}-1}, P_{\frac{n}{2}}, P_{(\frac{n}{2})'}, P_{\frac{n}{2}+1}\cdots, P_n\big\}$ and $[\omega]$ is an integral class. Assume for a fixed $i$ with $0\leq i \leq \frac{n}{2}$, the set of weights at each $P_k$ with $k\leq i$ and $k\geq n-i$ is
$$\big\{w_{kj} = \phi(P_j) -\phi(P_k)\big\}_{j\neq k, n-k},$$
and for each $k < i$ and $k > n-i$, $w_{kj}$ is a weight  between $P_k$ and $P_j$
for each $j\neq k, n-k$. Then $w_{ij}$ is a weight between $P_i$ and $P_j$ for each 
$j\neq i, n-i$, and $w_{n-i, j}$ is a weight between $P_{n-i}$ and $P_j$ for each 
$j\neq n-i, i$.
\end{lemma}

\section{from the sets of weights to the integral cohomology ring and total Chern class of $M$}\label{weightringc}

In this section, we use the sets of weights of the $S^1$-action to determine
the integral cohomology ring and total Chern class of the manifold, proving
$(1)$ and $(2)$ of Theorem  \ref{lw}. 

\smallskip

We will do this in a few steps.

\begin{lemma}\label{nprod<}
Let the circle act on a connected compact $2n$-dimensional symplectic manifold
$(M, \omega)$ with moment map $\phi\colon M\to\R$. Assume $[\omega]$ is a primitive integral class and 
$M^{S^1}=\big\{P_0,  \cdots, P_{\frac{n}{2}-1}, P_{\frac{n}{2}}, P_{(\frac{n}{2})'}, P_{\frac{n}{2}+1}\cdots, P_n\big\}$. 
Let $x=[\omega]$. 
Then the following conditions are equivalent:
\begin{enumerate}
\item There exist classes $y, z\in H^n(M; \Z)$ such that
the generators of $H^{2i}\big(M; \Z\big)$ with $0\leq 2i\leq n$ are $1$, $x$, $\cdots$, $x^{\frac{n}{2}-1}$, $y$ and $z$,  where $y+z=x^{\frac{n}{2}}$.
\item For any $i\in \big\{0, 1,  \cdots, \frac{n}{2}, (\frac{n}{2})'\big\}$, 
the product $\Lambda_i^-$ of the negative weights at $P_i$ is 
\begin{equation}\label{nprod<mid}
\Lambda_i^- = \prod_{\index(P_j) < \index (P_i)}\big(\phi(P_j)-\phi(P_i)\big).
\end{equation}
\item  For any $i\in \big\{\frac{n}{2}, (\frac{n}{2})', \frac{n}{2}+1, \cdots,  n\big\}$, the product $\Lambda_i^+$ of the positive weights at $P_i$ is 
\begin{equation}\label{pprod<mid}
\Lambda_i^+ = \prod_{\index(P_j) > \index (P_i)}\big(\phi(P_j)-\phi(P_i)\big).
\end{equation}
\end{enumerate}
\end{lemma}

\begin{proof}
Let us prove $(1)$ and $(2)$ are equivalent.
Let $\{\Tilde\alpha_i\}_{i\in\{0,  \cdots, \frac{n}{2}, (\frac{n}{2})', \cdots, n\}}$ be the basis of $H^*_{S^1}(M; \Z)$ as an $H^*(\CP^{\infty}; \Z)$-module as in Proposition~\ref{basis}. Let  $\{\alpha_i\}_{i\in\{0, \cdots, \frac{n}{2}, (\frac{n}{2})', \cdots, n\}}$ be the restriction of the above basis to ordinary cohomology, which is a basis of the ordinary cohomology of $M$.         

First, consider  $0\leq i <  \frac{n}{2}$. Since
$\prod_{j<i}\big(\ut + \phi(P_j)t\big)|_{P_j} =0\,\,\,\mbox{for all $j<i$},$
where $\ut$ is the class  in Lemma~\ref{ut}, by Corollary~\ref{cor2},
\begin{equation}\label{aa}
\prod_{j<i}\big(\ut + \phi(P_j)t\big) = a_i\,\Tilde\alpha_i, \,\,\, \mbox{where $a_i\in\Z$}.
\end{equation}
Restricting (\ref{aa}) to ordinary cohomology, we get
$[\omega]^i =a_i \alpha_i$.
Restricting (\ref{aa}) to $P_i$,  we get 
$\prod_{j<i}\big(\phi(P_j) - \phi(P_i)\big) = a_i \Lambda_i^-$.
Hence $\alpha_i = [\omega]^i$  if and only if 
(\ref{nprod<mid}) holds for $i$  (with $0\leq i <\frac{n}{2}$).

Similarly,  by Corollary~\ref{cor2},
\begin{equation}\label{bb}
\prod_{j<\frac{n}{2}}\big(\ut + \phi(P_j)t\big) = b\,\Tilde\alpha_{\frac{n}{2}}+ b'\,\Tilde\alpha_{(\frac{n}{2})'},\,\,\mbox{where $b, b'\in\Z$}.
\end{equation}
Restricting (\ref{bb}) to ordinary cohomology,  we get
$[\omega]^{\frac{n}{2}} = b\alpha_{\frac{n}{2}} + b'\,\alpha_{(\frac{n}{2})'}.$
Restricting (\ref{bb}) respectively to $P_{\frac{n}{2}}$ and $P_{(\frac{n}{2})'}$,  we get 
$$\prod_{j<\frac{n}{2}}\big(\phi(P_j)-\phi(P_{\frac{n}{2}})\big) = b\Lambda^-_{\frac{n}{2}},\,\,\,\mbox{and}\,\,\,
\prod_{j<\frac{n}{2}}\big(\phi(P_j)-\phi(P_{(\frac{n}{2})'})\big) = b' \Lambda^-_{(\frac{n}{2})'}.$$
Hence $[\omega]^{\frac{n}{2}} = \alpha_{\frac{n}{2}} + \,\alpha_{(\frac{n}{2})'}$  if and only if (\ref{nprod<mid})
holds for $i=\frac{n}{2}$ and $(\frac{n}{2})'$. We take $y=\alpha_{\frac{n}{2}}$ and
$z=\alpha_{(\frac{n}{2})'}$.

Using $-\phi$, we can get  the Poincar\'e dual classes $y'$ and $z'$ of the above
$y$ and $z$. The existence of the new basis $1, x, \cdots, x^{\frac{n}{2}-1}, y', z'$ is equivalent to $(1)$, and by the same argument as above is equivalent to $(3)$. 
\end{proof}

Next, we aim to prove Proposition~\ref{nprod>}. As a preparation, we first prove two lemmas.

\begin{lemma}\label{iresi+1}
Let the circle act on a connected compact $2n$-dimensional symplectic manifold
$(M, \omega)$ with moment map $\phi\colon M\to\R$. Assume 
$$M^{S^1}=\big\{P_0,  \cdots, P_{\frac{n}{2}-1}, P_{\frac{n}{2}}, P_{(\frac{n}{2})'}, P_{\frac{n}{2}+1}\cdots, P_n\big\}.$$ 
Then for any $i\in\big\{(\frac{n}{2})', \frac{n}{2}+1, \cdots, n-1\big\}$,
\begin{equation}\label{i+1>i}
\Tilde\alpha_i|_{P_{i+1}} = -\frac{\Lambda_{i+1}}{\Lambda_i^+}\prod_{j > i+1}\frac{\phi(P_j)-\phi(P_i)}{\phi(P_j)-\phi(P_{i+1})}t^i.
\end{equation}
Here, the $\Tilde\alpha_i$'s are the classes in Proposition~\ref{basis}, $\Lambda_{i+1}$ is the product of all the weights at $P_{i+1}$,
and we use the convention $(\frac{n}{2})' + 1 = \frac{n}{2}+1$ and $t^{(\frac{n}{2})'} = t^{\frac{n}{2}}$.
\end{lemma}

\begin{proof}
 For any $i\in\big\{(\frac{n}{2})', \frac{n}{2}+1, \cdots, n-1\big\}$, note that
$$\deg\big(\Tilde\alpha_i\cdot \prod_{j>i +1}\big(\ut +\phi(P_j) t\big)\big) < 2n.$$
Using Theorem~\ref{AB.BV} to integrate this class on $M$, we get
$$0 = \frac{\Tilde\alpha_i|_{P_i}\cdot\prod_{j>i +1}\left(\phi(P_j)-\phi(P_i)\right)}{\Lambda_i} 
+ \frac{\Tilde\alpha_i|_{P_{i+1}}\cdot\prod_{j>i +1}\left(\phi(P_j)-\phi(P_{i+1})\right)}{\Lambda_{i+1}}.$$
Solving this, we get (\ref{i+1>i}).
\end{proof}

\begin{lemma}\label{iresi+12}
Let the circle act on a connected compact $2n$-dimensional symplectic manifold
$(M, \omega)$ with moment map $\phi\colon M\to\R$. Assume $[\omega]$ is a primitive
integral class and
$M^{S^1}=\big\{P_0,  \cdots, P_{\frac{n}{2}-1}, P_{\frac{n}{2}}, P_{(\frac{n}{2})'}, P_{\frac{n}{2}+1}\cdots, P_n\big\}$. 
Then for any $i\in\big\{(\frac{n}{2})', \frac{n}{2}+1, \cdots, n-1\big\}$,
$[\omega]\alpha_i =\alpha_{i+1}$ if and only if
\begin{equation}\label{i+1>i2}
\Tilde\alpha_i|_{P_{i+1}} =\frac{\Lambda^-_{i+1}}{\phi(P_i)-\phi(P_{i+1})}t^i.
\end{equation}
Here, the $\Tilde\alpha_i$'s and $\alpha_i$'s are the classes in Proposition~\ref{basis}, and we use the convention $(\frac{n}{2})' + 1 = \frac{n}{2}+1$ and $t^{(\frac{n}{2})'} = t^{\frac{n}{2}}$.
\end{lemma}

\begin{proof}
For any $i\in\big\{(\frac{n}{2})', \frac{n}{2}+1, \cdots, n-1\big\}$, note that 
$$\deg\big(\big(\ut +\phi(P_i) t\big)\Tilde\alpha_i\big) = \deg (\Tilde\alpha_{i+1}),\,\,\,\mbox{and}$$
$$\big(\big(\ut +\phi(P_i) t\big)\Tilde\alpha_i \big)|_{P_j} =0, \,\,\,\, \forall \,\,\, P_j \,\,\,\mbox{with $\phi(P_j)<\phi(P_{i+1})$}.$$
By Corollary~\ref{cor2}, 
$\big(\ut +\phi(P_i) t\big)\Tilde\alpha_i = a_{i+1}\Tilde\alpha_{i+1}$ with $a_{i+1}\in\Z$. 
Restricting it respectively to ordinary cohomology and to $P_{i+1}$, we obtain the claim.
\end{proof}

\begin{proposition}\label{nprod>}
 Let the circle act on a connected compact $2n$-dimensional symplectic manifold
$(M, \omega)$ with moment map $\phi\colon M\to\R$. Assume $[\omega]$ is a primitive integral class and 
$M^{S^1}=\big\{P_0,  \cdots, P_{\frac{n}{2}-1}, P_{\frac{n}{2}}, P_{(\frac{n}{2})'}, P_{\frac{n}{2}+1}\cdots, P_n\big\}$. 
Let $x=[\omega]$. 
Then the following conditions are equivalent:
\begin{enumerate}
\item There exist classes $y, z\in H^n(M; \Z)$ such that
the generators of $H^*(M; \Z)$ are $1$, $x$, $\cdots$, $x^{\frac{n}{2}-1}$,
$y$, $z$, $xy=xz =\frac{1}{2} x^{\frac{n}{2}+1}$, $x^2y = x^2 z =\frac{1}{2} x^{\frac{n}{2}+2}$, $\cdots$, $x^{\frac{n}{2}}y = x^{\frac{n}{2}}z =\frac{1}{2}x^n$, where $x^{\frac{n}{2}} =y+z$.
\item  Lemma~\ref{nprod<}  $(2)$ holds, and
for any $i$ with $\frac{n}{2}+1\leq i\leq n$, the product of the negative weights at $P_i$ is
\begin{equation}\label{nprod>mid}
\Lambda_i^- =\frac{\prod_{\index (P_j) < \index (P_i)}\big(\phi(P_j)-\phi(P_i)\big)}{\big(\phi(P_{\frac{n}{2}}) -\phi(P_i)\big) + \big(\phi(P_{(\frac{n}{2})'})-\phi(P_i)\big)}.
\end{equation}
\item Lemma~\ref{nprod<}  $(3)$ holds, and
for any $i$ with $0\leq i\leq \frac{n}{2}-1$, the product of the positive weights at $P_i$ is
\begin{equation}\label{pprod<mid}
\Lambda_i^+ =\frac{\prod_{\index (P_j) > \index (P_i)}\big(\phi(P_j)-\phi(P_i)\big)}{\big(\phi(P_{\frac{n}{2}}) -\phi(P_i)\big) + \big(\phi(P_{(\frac{n}{2})'})-\phi(P_i)\big)}.
\end{equation}
\end{enumerate}
\end{proposition}

\begin{proof}
Let $\{\Tilde\alpha_i\}_{i\in\{0,  \cdots, \frac{n}{2}, (\frac{n}{2})', \cdots, n\}}$ and $\{\alpha_i\}_{i\in\{0,  \cdots, \frac{n}{2}, (\frac{n}{2})', \cdots, n\}}$ be respectively the basis of the equivariant and ordinary cohomology of $M$ as in Proposition~\ref{basis}. 

First, consider $i=\frac{n}{2}+1$. Since $\deg\left(\prod_{j \leq \frac{n}{2}}\big(\ut+\phi(P_j)t\big)\right) = n+2$, and
$\prod_{j \leq \frac{n}{2}}\big(\ut+\phi(P_j)t\big)|_{P_j} =0$ for all $j \leq \frac{n}{2}$, by Corollary~\ref{cor2},
\begin{equation}\label{QQ+1}
\prod_{j\leq\frac{n}{2}}\big(\ut+\phi(P_j)t\big)=a_{(\frac{n}{2})'} t\,\Tilde\alpha_{(\frac{n}{2})'} + b_{\frac{n}{2}+1}\,\Tilde\alpha_{\frac{n}{2}+1}, 
\end{equation}
where $a_{(\frac{n}{2})'}, b_{\frac{n}{2}+1}\in\Z$.
Restricting (\ref{QQ+1}) to ordinary cohomology, we get
$$[\omega]^{\frac{n}{2}+1} = b_{\frac{n}{2}+1}\alpha_{\frac{n}{2}+1}.$$
Restricting (\ref{QQ+1}) to $P_{(\frac{n}{2})'}$,  we get
$$\prod_{j \leq \frac{n}{2}}\big(\phi(P_j)-\phi(P_{(\frac{n}{2})'})\big) 
= a_{(\frac{n}{2})'}\Lambda^-_{(\frac{n}{2})'}.$$
Restricting  (\ref{QQ+1}) to $P_{\frac{n}{2}+1}$, we get
$$\prod_{j \leq \frac{n}{2}}\big(\phi(P_j)-\phi(P_{\frac{n}{2}+1})\big)=
a_{(\frac{n}{2})'}\Tilde\alpha_{(\frac{n}{2})'}|_{P_{\frac{n}{2}+1}} +b_{\frac{n}{2}+1}\,\Lambda^-_{\frac{n}{2}+1}.$$
Similarly, for any $i\geq\frac{n}{2}+2$, 
we can write
\begin{equation}\label{ii+1}
\prod_{\frac{1}{2}\index (P_j) <\, i-1}\left(\ut+ \phi(P_j)t\right)
=a_{i-1}\,t\,\Tilde\alpha_{i-1} + b_i\,\Tilde\alpha_i, 
\end{equation}
where $a_{i-1}, b_i\in\Z$.  Restricting (\ref{ii+1}) to ordinary cohomology, we get
$$[\omega]^i = b_i\alpha_i.$$
Restricting (\ref{ii+1}) to $P_{i-1}$, 
we get
$$\prod_{\frac{1}{2}\index (P_j) <\, i -1}\left(\phi(P_j)-\phi(P_{i-1})\right) =a_{i-1}\,\Lambda_{i-1}^-.$$
Restricting (\ref{ii+1}) to $P_i$, we get 
$$\prod_{\frac{1}{2}\index (P_j) <\, i -1}\left(\phi(P_j)-\phi(P_i)\right) =  a_{i-1}\Tilde\alpha_{i-1}|_{P_i} + b_i\Lambda_i^-.$$

Let $y=\alpha_{\frac{n}{2}}$ and  $z=\alpha_{(\frac{n}{2})'}$. We argue below that
$(1)$ and $(2)$ are equivalent.

Assume $(1)$ holds. Then Lemma~\ref{nprod<}  $(2)$ holds, $b_i=2$ for all  $\frac{n}{2}+1\leq i\leq n$, and (\ref{i+1>i2}) holds.
Using (\ref{nprod<mid}), we can get $a_{(\frac{n}{2})'}$, then get $\Lambda^-_{\frac{n}{2}+1}$. Inductively, we can
get $a_{i-1}$, and then get $\Lambda_i^-$ for all $i\geq \frac{n}{2}+2$. 

Conversely, assume $(2)$ holds. Then by Lemma~\ref{nprod<}, we have the generators
of $H^{2i}(M; \Z)$ for $0\leq 2i\leq n$ as in $(1)$ and $x^{\frac{n}{2}}=y+z$, and
Lemma~\ref{nprod<}  $(3)$ holds. By (\ref{i+1>i}), we get that (\ref{i+1>i2}) holds. Then Lemma~\ref{iresi+12} implies that $x^k z =\alpha_{\frac{n}{2}+k}$ for all $1\leq k\leq\frac{n}{2}$. Moreover, from the above equalities, we can solve $a_{(\frac{n}{2})'}$ and all $a_{i-1}$, and then get $b_i=2$ for all  $\frac{n}{2}+1\leq i\leq n$. Note that
$x^ky+ x^k z = x^{\frac{n}{2}+k}$ for all $1\leq k\leq\frac{n}{2}$. Hence 
$x^ky = x^k z = \frac{1}{2}x^{\frac{n}{2}+k} =\alpha_{\frac{n}{2}+k}$  are generators of $H^{n+2k}(M; \Z)$ for all  $1\leq k\leq\frac{n}{2}$.

Using $-\phi$ as a Morse function, we get the Poincar\'e dual of the classes
in $(1)$, they form a new basis (satisfying the same relations). 
The existence of the new basis is equivalent to $(1)$, and by the same arguments as above, is equivalent to $(3)$.
\end{proof}

Next, having Propositions~\ref{01i2} and \ref{nprod>}, we look at the ring structure of the manifold.

\begin{proposition}\label{weights-ring}
Let the circle act on a  connected compact $2n$-dimensional symplectic manifold
$(M, \omega)$ with moment map $\phi\colon M\to\R$. Assume $[\omega]$ is an integral class and
$M^{S^1}=\big\{P_0,  \cdots, P_{\frac{n}{2}-1}, P_{\frac{n}{2}}, P_{(\frac{n}{2})'}, P_{\frac{n}{2}+1}\cdots, P_n\big\}$.
Assume the integer $\phi(P_{n-1})-\phi(P_0)=\phi(P_n)-\phi(P_1)$ occurs as a weight of the $S^1$-action at some fixed point.
Then Proposition~\ref{nprod>}  $(1)$ holds, moreover,
$$x^{\frac{n}{2}}y=x^{\frac{n}{2}}z = yz,\, \,y^2=z^2=0,\,\,\,\mbox{when $\dim(M)=4m$ with $m$ odd},\,\,\,\mbox{and}$$
$$x^{\frac{n}{2}}y=x^{\frac{n}{2}}z =y^2=z^2, \,\, yz=0, \,\,\,\mbox{when $\dim(M)=4m$ with $m$  even}.$$
Hence, as rings, $H^*(M; \Z)\cong  H^*\big(\Gt_2(\R^{n+2}); \Z\big)$ with $n\geq 2$ even.
\end{proposition}

\begin{proof}
By Proposition~\ref{01i2}, Proposition~\ref{nprod>} $(2)$ (and $(3)$) holds. Hence Proposition~\ref{nprod>}  $(1)$ holds. It remains to determine the relations of $y^2$, $z^2$ and $yz$ with the top degree generators. First, since $x^{\frac{n}{2}} = y +z$,
we have $x^{\frac{n}{2}}y = y^2 + yz$ and $x^{\frac{n}{2}}z = z^2 + yz$. Since 
$x^{\frac{n}{2}}y=x^{\frac{n}{2}}z$, we have $y^2 = z^2$.

Recall that $y=\alpha_{\frac{n}{2}}$ and $z=\alpha_{(\frac{n}{2})'}$, where  
$\alpha_{\frac{n}{2}}$ and $\alpha_{(\frac{n}{2})'}$ are respectively the restrictions of  $\Tilde\alpha_{\frac{n}{2}}$ and $\Tilde\alpha_{(\frac{n}{2})'}$ to ordinary cohomology.
For any $k \geq \frac{n}{2}+1$,
note that
$$\Big(\Tilde\alpha_{\frac{n}{2}}\cdot\prod_{j \in\{\frac{n}{2}, \frac{n}{2}+1,\, \cdots,\, k-1\}}\big(\ut +\phi(P_j) t\big)\Big)\arrowvert_{P_j} = 0,\,\,\,\forall \,\, P_j \,\,\,\mbox{with $\phi(P_j) < \phi(P_k)$}.$$
By Corollary~\ref{cor2}, 
$$\Tilde\alpha_{\frac{n}{2}}\cdot\prod_{j \in\{\frac{n}{2}, \frac{n}{2}+1,\, \cdots, \,k-1\}}\big(\ut +\phi(P_j) t\big)=a_k\Tilde\alpha_k, \,\,\,\mbox{with $a_k\in\Z$}.$$
By restricting it to ordinary cohomology, we see that $a_k=1$. Then restricting it to $P_k$, we get
$$\Tilde\alpha_{\frac{n}{2}}|_{P_k} = \frac{\Lambda_k^-}{\prod_{j \in\{\frac{n}{2}, \frac{n}{2}+1, \,\cdots,\, k-1\}}\big(\phi(P_j)-\phi(P_k)\big)}.$$
Similarly, we have
$$\Tilde\alpha_{(\frac{n}{2})'}|_{P_k} = \frac{\Lambda_k^-}{\prod_{j \in\{(\frac{n}{2})', \frac{n}{2}+1,\, \cdots, \,k-1\}}\big(\phi(P_j)-\phi(P_k)\big)},\,\,\,\forall\,\, k \geq \frac{n}{2}+1.$$
By Theorem~\ref{AB.BV}, 
$$\int_M y^2 =\int_M \big(\Tilde\alpha_{\frac{n}{2}}\big)^2= \frac{\Lambda^-_{\frac{n}{2}}    \Lambda^-_{\frac{n}{2}}}{\Lambda_{\frac{n}{2}}} + \sum_{k\geq\frac{n}{2}+1}\frac{(\Tilde\alpha_{\frac{n}{2}})^2|_{P_k}}{\Lambda_k} = \frac{\Lambda^-_{\frac{n}{2}}}{\Lambda^+_{\frac{n}{2}}} + \sum_{k\geq\frac{n}{2}+1}\frac{(\Tilde\alpha_{\frac{n}{2}})^2|_{P_k}}{\Lambda_k}, \,\,\,\mbox{and}$$
$$\int_M yz = \int_M \Tilde\alpha_{\frac{n}{2}} \Tilde\alpha_{(\frac{n}{2})'} =\sum_{k\geq\frac{n}{2}+1}\frac{  \Tilde\alpha_{\frac{n}{2}}|_{P_k}\cdot\Tilde\alpha_{(\frac{n}{2})'}|_{P_k}}{\Lambda_k}.$$ 
Since the integrals $\int_M y^2$ and  $\int_M yz$ should not depend on the moment map values of the fixed points,  we may assume that $\phi(P_{\frac{n}{2}})=\phi(P_{(\frac{n}{2})'})$. 
Then by the expressions above,
$\Tilde\alpha_{\frac{n}{2}}|_{P_k}=\Tilde\alpha_{(\frac{n}{2})'}|_{P_k},\,\,\forall \,\,k\geq\frac{n}{2}+1$,
and by Proposition~\ref{01i2}, $\Lambda^-_{\frac{n}{2}} = (-1)^{\frac{n}{2}}\Lambda^+_{\frac{n}{2}}$. These together give $$\int_M y^2 - \int_M yz =  (-1)^{\frac{n}{2}}.$$
While $x^{\frac{n}{2}}y = y^2 + yz$ gives
$$\int_M y^2 + \int_M yz =\int_M x^{\frac{n}{2}}y =1.$$  
Hence 
$$\int_M y^2 = \frac{1+(-1)^{\frac{n}{2}}}{2}.$$
\end{proof}

Finally, we use the weights in Proposition~\ref{01i2} to determine $c(M)$.

\begin{proposition}\label{weights-tc}
Let the circle act on a connected compact $2n$-dimensional symplectic manifold
$(M, \omega)$ with moment map $\phi\colon M\to\R$. Assume  $[\omega]$ is an integral class and
$M^{S^1}=\big\{P_0, \cdots, P_{\frac{n}{2}-1}, P_{\frac{n}{2}}, P_{(\frac{n}{2})'}, P_{\frac{n}{2}+1}, \cdots, P_n\big\}$.
Assume $\phi(P_{n-1})-\phi(P_0)=\phi(P_n)-\phi(P_1)$ holds and this integer occurs as a weight of the $S^1$-action at some fixed point.
Then the total Chern class of $M$ is $c(M)=\frac{(1+[\omega])^{n+2}}{1+2[\omega]}$, hence is
isomorphic to $c\big(\Gt_2(\R^{n+2})\big)$.
\end{proposition}

\begin{proof}
By the injectivity theorem (\cite{K, TW}), the restriction map
$$H^*_{S^1}(M; \Z)\to H^*_{S^1}(M^{S^1}; \Z)$$
 is injective. Let 
$I=\{0,  \cdots, \frac{n}{2}, (\frac{n}{2})', \cdots, n\}$.
Consider the class 
$$\alpha=\frac{\prod_{j\in I}\big(1+\ut +\phi(P_j)t\big)}{1+\big(\ut+\phi(P_0)t\big) +\big(\ut+\phi(P_n)t\big)}.$$
By Proposition~\ref{01i2}, 
$\big(\phi(P_0) - \phi(P_i)\big) +\big(\phi(P_n) - \phi(P_i)\big)= \phi(P_{n-i})-\phi(P_i)$ for all
$i\in I$.  Using this, and  Proposition~\ref{01i2}, we can check that 
$$\alpha|_{P_i}=c^{S^1}(M)|_{P_i},\,\,\,\forall \,\, i\in I.$$
Hence $c^{S^1}(M) = \alpha$.  Restricting this to ordinary cohomology, we obtain our claim.
\end{proof}

\end{document}